\documentclass[12pt,draft,reqno]{amsart}
\usepackage{amsmath}
\usepackage{amssymb}
\usepackage{amsthm}
\usepackage{bm}
\numberwithin{equation}{section}
\newtheorem{theorem}{Theorem}[section]
\newtheorem{lemma}[theorem]{Lemma}
\newtheorem{proposition}[theorem]{Proposition}
\newtheorem{corollary}[theorem]{Corollary}

\newtheorem{thm}{Theorem}

\newtheorem{dfn}{Definition}

\theoremstyle{definition}
\newtheorem*{remark}{Remark}

\allowdisplaybreaks
\begin{document}
\title[The values of zeta functions at integers]{The values of zeta functions composed by the Hurwitz and periodic zeta functions at integers}
\author[T.~Nakamura]{Takashi Nakamura}
\address[T.~Nakamura]{Department of Liberal Arts, Faculty of Science and Technology, Tokyo University of Science, 2641 Yamazaki, Noda-shi, Chiba-ken, 278-8510, Japan}
\email{nakamuratakashi@rs.tus.ac.jp}
\urladdr{https://sites.google.com/site/takashinakamurazeta/}
\maketitle

\begin{abstract}
For $s \in {\mathbb{C}}$ and $0 < a <1$, let $\zeta (s,a)$ and ${\rm{Li}}_s (e^{2\pi ia})$ be the Hurwitz and periodic zeta functions, respectively. For $0 < a \le 1/2$, put $Z(s,a) := \zeta (s,a) + \zeta (s,1-a)$, $P(s,a) := {\rm{Li}}_s (e^{2\pi ia}) + {\rm{Li}}_s (e^{2\pi i(1-a)})$, $Y(s,a) := \zeta (s,a) - \zeta (s,1-a)$ and $O(s,a):= -i \bigl( {\rm{Li}}_s (e^{2\pi ia}) - {\rm{Li}}_s (e^{2\pi i(1-a)}) \bigr)$. 

Let $n \ge 0$ be an integer and $b := r/q$, where $q>r>0$ are coprime integers. In this paper, we prove that the values $Z(-n,b)$, $\pi^{-2n-2} P(2n+2,b)$, $Y(-n,b)$ and $\pi^{-2n-1} O(2n+1,b)$ are rational numbers, in addition, $\pi^{-2n-2} Z(2n+2,b)$, $P(-n,b)$, $\pi^{-2n-1} Y(2n+1,b)$ and $O(-n,b)$ are polynomials of $\cos (2\pi/q)$ and $\sin (2\pi/q)$  with rational coefficients. Furthermore, we show that $Z(-n,a)$, $\pi^{-2n-2} P(2n+2,a)$, $Y(-n,a)$ and $\pi^{-2n-1} O(2n+1,a)$ are polynomials of $0<a<1$ with rational coefficient, in addition, $\pi^{-2n-2} Z(2n+2,a)$, $P(-n,a)$, $\pi^{-2n-1} Y(2n+1,a)$ and $O(-n,a)$ are rational functions of $\exp (2 \pi ia)$ with rational coefficients. Note that the rational numbers, polynomials and rational functions mentioned above are given explicitly. 

Moreover, we show that $P(s,a) \equiv 0$ for all $ 0 < a < 1/2$ if and only if $s$ is a negative even integer. We also prove similar assertions for $Z(s,a)$, $Y(s,a)$, $O(s,a)$ and so on. In addition, we prove that the function $Z(s,|a|)$ appears as the spectral density of some stationary self-similar Gaussian distributions.
\end{abstract}

\setcounter{tocdepth}{1}
\tableofcontents

\section{Introduction and statement of main results}

\subsection{Special values of the Riemann zeta function}
For a complex variable $s = \sigma +it$, where $\sigma,t \in {\mathbb{R}}$, the Riemann zeta function is defined by
$$
\zeta (s) := \sum_{n=1}^\infty \frac{1}{n^s}, \qquad \sigma >1.
$$
According to the integral representation
$$
\zeta (s) = \frac{e^{-i\pi s} \Gamma (1-s)}{2 \pi i} \int_C \frac{z^{s-1}}{e^z-1} dz,
$$ 
where the contour $C$ starts at infinity on the positive real axis, encircles the origin once in the positive direction, excluding the points $\pm 2\pi i$, $\pm 4\pi i,\ldots$, and returns to the positive infinity, we can see that $\zeta (s)$ is a meromorphic function with a simple pole at $s=1$ with residue $1$. Moreover, the function $\zeta (s)$ satisfies the functional equation 
\begin{equation}\label{eq:RZfe}
\zeta(1-s) =\frac{2\Gamma (s)}{(2\pi )^s} \cos \Bigl( \frac{\pi s}{2} \Bigr) \zeta(s)
\end{equation}
from the integral representation above. The functional equation and integral representation derive the following (see for example \cite[Section 12.12]{Apo} and \cite[Section 2.4]{Tit}). Note that the $n$-th Bernoulli number $B_n$ is defined in Section 2.1.

\begin{thm}
For $n \in {\mathbb{N}}$, we have
$$
\zeta (2n) = (-1)^{n+1} \frac{(2\pi)^{2n}}{(2n)!} B_{2n}. 
$$
For every integer $n \ge0$, it holds that
$$
\zeta (-n) = -\frac{B_{n+1}}{n+1}.
$$
\end{thm}

Let $\chi$ be a Dirichlet character and $L(s,\chi)$ be the Dirichlet $L$-function associated to the character $\chi$. And let $\chi$ be a primitive character and $n$ be a natural number. Then it is known that $\pi^{-n} L(n,\chi)$ is written by a Gauss sum and the generalized Bernoulli number if $\chi (-1) = (-1)^n$ (see for example \cite[Theorem 9.6]{AIK}). Furthermore, it is also known that $L(-n,\chi)$ is written by a generalized Bernoulli number when $n$ is a non-negative integer (see \cite[Theorem 9.10]{AIK}). Obviously, these facts are an analogue of Theorem A. In general, there are no such explicit evaluation formulas at integers for automorphic $L$-function, Epstein zeta functions, the prime zeta function and so on. Hence, there are few zeta or $L$-functions of which the values at both positive and negative integers are expressed by $\pi$ and (generalized) Bernoulli numbers. 

\subsection{The Hurwitz and periodic zeta functions}

The Hurwitz zeta function $\zeta (s,a)$ is defined by the series
$$
\zeta (s,a) := \sum_{n=0}^\infty \frac{1}{(n+a)^s}, \qquad \sigma >1, \quad 0<a \le 1.
$$
The function $\zeta (s,a)$ is  meromorphic and has a simple pole at $s=1$ whose residue is $1$ (see for instance \cite[Section 12]{Apo}). 
Next, we define the periodic zeta function by  
$$
{\rm{Li}}_s (e^{2\pi ia}) := \sum_{n=1}^\infty \frac{e^{2\pi ina}}{n^s}, \qquad \sigma >1, \quad 0<a \le 1
$$
(see for example \cite[Exercise 12.2]{Apo}). The periodic zeta function ${\rm{Li}}_s (e^{2\pi ia})$ with $0<a<1$ is analytically continuable to the whole complex plane since the Dirichlet series of ${\rm{Li}}_s (e^{2\pi ia})$ converges uniformly in each compact subset of the half-plane $\sigma >0$ when $0<a<1$ (see for example \cite[p.~20]{LauGa}). Note that $\zeta (-n,a)$ and ${\rm{Li}}_{-n} (e^{2\pi ia})$ , where $n$ is a non-negative integer, are written by the Bernoulli number and Stirling number of the second kind, respectively (see Lemmas \ref{lemB2} and \ref{eq:SKKY2}). However, there are no such formulas for $\zeta (n,a)$ and ${\rm{Li}}_{n} (e^{2\pi ia})$ , where $n$ is a integer greater than $1$. 

For $0 <a \le 1/2$, let
\begin{equation*}
\begin{split}
&Z(s,a) := \zeta (s,a) + \zeta (s,1-a), \qquad P(s,a) := {\rm{Li}}_s (e^{2\pi ia}) + {\rm{Li}}_s (e^{2\pi i(1-a)}), \\
&2Q(s,a):= Z(s,a) + P(s,a) = \zeta (s,a) + \zeta (s,1-a) + {\rm{Li}}_s (e^{2\pi ia}) + {\rm{Li}}_s (e^{2\pi i(1-a)}),
\end{split}
\end{equation*}
\begin{equation*}
\begin{split}
&Y(s,a) := \zeta (s,a) - \zeta (s,1-a), \qquad O(s,a):= -i \bigl( {\rm{Li}}_s (e^{2\pi ia}) - {\rm{Li}}_s (e^{2\pi i(1-a)}) \bigr) ,\\
&2X(s,a) := Y(s,a) + O(s,a) = \zeta (s,a) - \zeta (s,1-a) -i \bigl( {\rm{Li}}_s (e^{2\pi ia}) - {\rm{Li}}_s (e^{2\pi i(1-a)}) \bigr) .
\end{split}
\end{equation*}
It should be noted that the functions $Y(s,a)$, $O(s,a)$ and $X(s,a)$ are entire (see \cite[Section 3.3]{NPCZ}). We remark that one has $Y(s,1/2) \equiv O(s,1/2) \equiv X(s,1/2) \equiv 0$.

In \cite[Section 1.2]{NPCZ} and \cite[Section 1.1]{NPRCZ}, the following are shown.

\begin{thm}\label{th:m1}
All real zeros of the function $Z(s,a)$ are simple and only at the non-positive even integers if and only if $1/4 \le a \le 1/2$.

Moreover, all real zeros of the function $P(s,a)$ are simple and only at the negative even integers if and only if $1/4 \le a \le 1/2$.
\end{thm}

\begin{thm}\label{th:mQ1}
All real zeros of the quadrilateral zeta function $Q(s,a)$ are simple and only at the negative even integers if and only if $a_0 < a \le 1/2$, where $a_0 = 0.1183751396...$ satisfies $Z(1/2, a_0) = P(1/2,a_0) = Q(1/2, a_0) = 0$.
\end{thm}

\begin{thm}\label{th:OY1}
All real zeros of the functions $Y(s,a)$, $O(s,a)$ or $X(s,a)$ with $0 < a < 1/2$ are simple and only at the negative odd integers. 
\end{thm}

It should be emphasised that from the theorems above, the gap between consecutive real zeros of $Y(s,a)$, $O(s,a)$, $X(s,a)$ with $0 < a < 1/2$,  $Z(s,a)$ and $P(s,a)$ with $1/4 \le a \le 1/2$ and $Q(s,a)$ with $a_0 < a \le 1/2$ is always $2$, in other words, the the gaps do not depend on $a$ just like the Riemann zeta function $\zeta (s)$.

\subsection{Main results}
In the present paper, we investigate the values of zeta functions $Z(s,a)$, $P(s,a)$, $Q(s,a)$, $Y(s,a)$, $O(s,a)$ and $X(s,a)$ at integers (see Theorems \ref{th:Ber}, \ref{th:poly}, \ref{th:d1} and \ref{th:d2}). Moreover, we show that the zeta function $Z(s,a)$ is related to some stationary self-similar Gaussian distributions in Proposition \ref{th:apSi1}.

When $a =r/q$ is a rational number, we have the following as an analogue of the fact $\zeta (2n) \in {\mathbb{Q}} \pi ^{2n}$ proved by Theorem A. Note that the explicit evaluation formulas for the special vales below are given in Section 2.2.
\begin{theorem}\label{th:Ber}
Let $n$ be a non-negative integer and $q>r>0$ be coprime integers. Then 
$$
Z(-n, r/q), \quad Y(-n, r/q), \quad \pi^{-2n-2} P(2n+2, r/q), \quad \pi^{-2n-1} O(2n+1, r/q)
$$
are rational numbers. Moreover, 
\begin{equation*}
\begin{split}
&\pi^{-2n-2} Z(2n+2, r/q) , \quad P(-n, r/q), \quad \pi^{-2n-1} Y(2n+1, r/q), \quad O(-n, r/q), \\
&\pi^{-2n-2} Q(2n+2, r/q), \quad Q(-n, r/q), \quad \pi^{-2n-1} X(2n+1, r/q), \quad X(-n, r/q)
\end{split}
\end{equation*}
are elements of the polynomial ring ${\mathbb{Q}} [\cos (2\pi/q), \sin (2\pi/q)]$.
\end{theorem}

Next we prove the following when $0 < a < 1/2$ is irrational. It should be emphasised that the polynomials and rational functions in the theorem below are given explicitly in Section 3.2. 
\begin{theorem}\label{th:poly}
Let $n$ be a non-negative integer. Then 
$$
Z(-n,a), \quad \pi^{-2n-2}P(2n+2,a),  \quad Y(-n,a) , \quad \pi^{-2n-1} O(2n+1,a)
$$ 
are polynomials with rational coefficients of $0 < a < 1/2$. Furthermore, 
\begin{equation*}
\begin{split}
&\pi^{-2n-2} Z(2n+2,a), \quad P(-n,a), \quad \pi^{-2n-1} Y(2n+1,a) \quad O(-n,a) \\
&\pi^{-2n-2} Q(2n+2, a), \quad Q(-n, a), \quad \pi^{-2n-1} X(2n+1, a), \quad X(-n, a)
\end{split}
\end{equation*}
are rational functions with rational coefficients of $\exp (2 \pi ia)$. 
\end{theorem}

We can see that $P(s,a)$ identically vanishes for all $0<a<1/2$ if $s$ is a negative even integer by the functional equation of $P(s,a)$ (see Theorem B and Lemma \ref{lem:fe1}). The next theorem implies that $P(s,a) \equiv 0$ for all $0<a<1/2$ only if $s$ is a negative even integer.

\begin{theorem}\label{th:d1}
Let $s \ne 1$. Then we have
$$
Z(s,a) \equiv 0 \quad \mbox{ for all } \quad 0 < a < 1/2
$$
if and only if $s$ is a non-positive even integer. Furthermore it holds that 
$$
Q(s,a) \equiv 0 \quad \mbox{ for all } \quad 0 < a < 1/2
$$
if and only if $s$ is a negative even integer. 

Next let $s \in {\mathbb{C}}$. Then one has 
$$
P(s,a) \equiv 0 \quad \mbox{ for all } \quad 0 < a < 1/2
$$
if and only if $s$ is a negative even integer. Moreover, we have
$$
Y(s,a) \equiv 0 \quad \mbox{ for all } \quad 0 < a < 1/2
$$
if and only if $s$ is an odd negative integer. The same statement holds for the zeta functions $O(s,a)$ and $X(s,a)$.
\end{theorem}

On the other hand, we have the following for $\zeta(s,a)$ and ${\rm{Li}}_s (e^{2\pi ia})$. 
\begin{theorem}\label{th:d2}
For any $1 \ne s \in {\mathbb{C}}$, there exists $0 < a < 1/2$ such that
$$
\zeta (s,a) \ne 0.
$$
For any $s \in {\mathbb{C}}$, there is $0 < a < 1/2$ such that
$$
{\rm{Li}}_s (e^{2\pi ia}) \ne 0. 
$$
\end{theorem}

Moreover, we have the following proposition which implies that $Z(s,|a|)$ appears as the spectral density of some stationary self-similar Gaussian distributions (for details see Appendix or \cite[Section 1]{Sinai}). 
\begin{proposition}\label{th:apSi1}
Let $P$ be a one-dimensional stationary Gaussian distribution on $X$ with ${\mathbb{E}} x_l =0$. Then the distribution $P$ is an s.s.d, if and only if its spectral density $\rho_\lambda(\alpha)$ has the form
$$
\rho_\lambda(\alpha):= C | e^{2\pi i \alpha} -1|^2 Z(\lambda+1, |\alpha|),
\qquad -1/2 \le \alpha \le 1/2, 
$$
where $C>0$ is a constant.
\end{proposition}

In Section 2 and 3, we prove Theorems \ref{th:Ber} and \ref{th:poly}, respectively. We prove Theorems \ref{th:d1} and \ref{th:d2} in Section 4. In Section 5, we prove Proposition \ref{th:apSi1}.

\section{Proof of Theorem \ref{th:Ber}}

\subsection{Bernoulli polynomials and functional equations}
We denote by $B_n (t)$ the Bernoulli polynomial of order $n$ defined as
$$
\frac{ze^{tz}}{e^z-1} = \sum_{n=0}^{\infty} B_n (t) \frac{z^n}{n!} .
$$
The first few are:
$$
B_0 (t) =1, \qquad B_1 (t) = t - \frac{1}{2}, \qquad B_2 (t) = t^2 - t + \frac{1}{6},
$$
$$
B_3 (t) = t^3 - \frac{3}{2} t^2 + \frac{1}{2} t, \qquad B_4 (t) = t^4 - 2t^3 + t^2 - \frac{1}{30}  .
$$
The following equation is well-known (see for example \cite[Exercise 12.11]{Apo}). 
\begin{equation}\label{eq:ab1-a}
B_n (1-a) = (-1)^n B_n (a), \qquad n \ge 0.
\end{equation}
And we define the $n$-th Bernoulli number $B_n$ by
$$
B_n : = B_n(1).
$$
The following are well-known (see for instance \cite[Theorems 12.19 and 12.13]{Apo}).
\begin{lemma}\label{lemB1}
If $k \in {\mathbb{N}}$ and $0 < a < 1$, one has
$$
B_{2k} (a) = (-1)^{k+1} \frac{2(2k)!}{(2\pi)^{2k}} \sum_{m=1}^\infty \frac{\cos 2\pi ma}{m^{2k}}, \quad
B_{2k-1} (a) = (-1)^k \frac{2(2k-1)!}{(2\pi)^{2k-1}} \sum_{m=1}^\infty \frac{\sin 2\pi ma}{m^{2k-1}}.
$$
\end{lemma}
\begin{lemma}\label{lemB2}
For every integer $n \ge0$, it holds that
$$
\zeta (-n,a) = - \frac{B_{n+1}(a)}{n+1}.
$$
\end{lemma}

Next we quote the functional equations for $\zeta (s,a)$ and ${\rm{Li}}_s (e^{2\pi ia})$ (see \cite[Theorem 12.6 and Exercise 12.2]{Apo}) and $Z(s,a)$, $P(s,a)$, $Y(s,a)$, $Q(s,a)$ and $X(s,a)$ (see \cite[Sections 3.3 and 4.2]{NPCZ}).
\begin{lemma}\label{lem:fe1}
It holds that
\begin{equation*}
\begin{split}
&\zeta (1-s, a) = 
\frac{\Gamma (s)}{(2\pi)^s} \Bigl( e^{-\pi is/2} {\rm{Li}}_s (e^{2\pi ia}) + e^{\pi is/2} {\rm{Li}}_s (e^{2\pi i(1-a)}) \Bigr),\\
&\,\,\, {\rm{Li}}_{1-s} (e^{2\pi ia}) = 
\frac{\Gamma (s)}{(2\pi )^s} \Bigl( e^{\pi is/2} \zeta (s,a) + e^{-\pi is/2} \zeta (s,1-a) \Bigr),
\end{split}
\end{equation*}
\begin{equation*}
\begin{split}
&Z(1-s,a) = \frac{2\Gamma (s)}{(2\pi )^s} \cos \Bigl( \frac{\pi s}{2} \Bigr) P(s,a), \qquad
P(1-s,a) = \frac{2\Gamma (s)}{(2\pi )^s} \cos \Bigl( \frac{\pi s}{2} \Bigr) Z(s,a), \\
&Y(1-s,a) = \frac{2\Gamma (s)}{(2\pi )^s} \sin \Bigl( \frac{\pi s}{2} \Bigr) O(s,a), \qquad 
O(1-s,a) = \frac{2\Gamma (s)}{(2\pi )^s} \sin \Bigl( \frac{\pi s}{2} \Bigr) Y(s,a), \\
&Q(1-s,a) = \frac{2\Gamma (s)}{(2\pi )^s} \cos \Bigl( \frac{\pi s}{2} \Bigr) Q(s,a), \qquad
X(1-s,a) = \frac{2\Gamma (s)}{(2\pi )^s} \sin \Bigl( \frac{\pi s}{2} \Bigr) X(s,a).
\end{split}
\end{equation*}
\end{lemma}

\subsection{Proof of Theorem \ref{th:Ber}}
In this subsection, we give explicit evaluation formulas for $Z(-n,b)$, $P(2n+2,b)$, $Y(-n,b)$, $O(2n+1,b)$, $Z(2n+2,b)$, $P(-n,b)$, $Y(2n+1,b)$ and $O(-n,b)$, where $n$ is a non-positive integer, $q>r>0$ are coprime integers and $b := r/q$, which prove Theorem \ref{th:Ber}. 

The next well-known formula plays an important role in the proof of Theorem \ref{th:Ber}.
\begin{lemma}\label{lem:invmul}
Let $r,q \in {\mathbb{N}}$ be coprime and $q>r>0$. The one has
\begin{equation}\label{eq:mul1}
{\rm{Li}}_s (e^{2\pi irm/q}) = q^{-s} \sum_{m=1}^q e^{2\pi irm/q} \zeta(s,m/q), \qquad s \in {\mathbb{C}}.
\end{equation}
\end{lemma}
\begin{proof}
For readers convenience, we write the proof. Let $|z|=1$ and $\sigma >1$. Then it is easy to see that
$$
\sum_{l=1}^\infty \frac{z^l}{l^s} = \sum_{m=1}^q \sum_{l=0}^\infty \frac{z^{ql+m}}{(ql+m)^s} =
q^{-s} \sum_{m=1}^q z^m \sum_{l=0}^\infty \frac{z^{ql}}{(l+m/q)^s}.
$$
By putting $z = e^{2\pi ir/q}$, we have (\ref{eq:mul1}). 
\end{proof}

\begin{proposition}\label{pro:raZYPOn}
Let $r,q \in {\mathbb{N}}$ be coprime and $q>r>0$. Then, for $n \in {\mathbb{N}}$, we have
\begin{equation*}
\begin{split}
&Z(-n,r/q) = \frac{(-1)^n-1}{n+1} B_{n+1}(r/q), \qquad Y(-n,r/q) = \frac{(-1)^{n+1}-1}{n+1} B_{n+1}(r/q),\\
&\qquad \qquad P(-n,r/q) = -\frac{2 q^n}{n+1} \sum_{m=1}^q \cos (2\pi rm/q) B_{n+1} (m/q), \\ 
&\qquad \qquad O(-n,r/q) = -\frac{2 q^n}{n+1} \sum_{m=1}^q \sin (2\pi rm/q) B_{n+1} (m/q) .
\end{split}
\end{equation*}
\end{proposition}
\begin{proof}
By using Lemma \ref{lemB2}, we have
$$
Z(-n,r/q) = \zeta (-n,r/q) + \zeta(-n,1-r/q) = - \frac{B_{n+1}(r/q)}{n+1} - \frac{B_{n+1}(1-r/q)}{n+1}.
$$
Hence we obtain the first equation of Proposition \ref{pro:raZYPOn} from (\ref{eq:ab1-a}). Similarly, we have
$$
Y(-n,r/q) = \zeta (-n,r/q) - \zeta(-n,1-r/q) = - \frac{B_{n+1}(r/q)}{n+1} + \frac{B_{n+1}(1-r/q)}{n+1}
$$
which implies the second equation of Proposition \ref{pro:raZYPOn}. From (\ref{eq:mul1}), it holds that
\begin{equation*}
\begin{split}
&\quad q^s P(s,r/q) = q^s \Bigl( {\rm{Li}}_s (e^{2\pi irm/q}) + {\rm{Li}}_s (e^{2\pi i(q-r)m/q}) \Bigr) \\ 
&= \sum_{m=1}^q e^{2\pi irm/q} \zeta(s,m/q) + \sum_{m=1}^q e^{-2\pi irm/q} \zeta(s,m/q) 
= 2 \sum_{m=1}^q \cos (2\pi rm/q) \zeta (s,m/q) .
\end{split}
\end{equation*}
By (\ref{eq:mul1}), we similarly obtain
\begin{equation*}
\begin{split}
&\quad i q^s O(\sigma,r/q) = i q^s \Bigl( {\rm{Li}}_s (e^{2\pi irm/q}) - {\rm{Li}}_s (e^{2\pi i(q-r)m/q}) \Bigr) \\
&= \sum_{m=1}^q e^{2\pi irm/q} \zeta(s,m/q) - \sum_{m=1}^q e^{-2\pi irm/q} \zeta(s,m/q)  =
2 i \sum_{m=1}^q \sin (2\pi rm/q) \zeta (s,m/q) .
\end{split}
\end{equation*}
Hence we have the third and fourth formulas of Proposition \ref{pro:raZYPOn} from Lemma \ref{lemB2}.
\end{proof}

The next proposition is proved by Proposition \ref{pro:raZYPOn} above and the functional equations in Lemma \ref{lem:fe1}. 
\begin{proposition}\label{pro:raZYPOp}
Let $r,q \in {\mathbb{N}}$ be coprime. Then, for $n \in {\mathbb{N}}$, we have
\begin{equation*}
\begin{split}
&\qquad \qquad \qquad \quad Z(2n,r/q) = (-1)^{n+1}  q^{2n-1} \frac{(2\pi)^{2n}}{(2n)!} \sum_{m=1}^q \cos (2\pi rm/q) B_{2n} (m/q), \\
&\qquad \qquad \quad Y(2n-1,r/q) = (-1)^{n} q^{2n-2} \frac{(2\pi)^{2n-1}}{(2n-1)!} \sum_{m=1}^q \sin (2\pi rm/q) B_{2n-1} (m/q), \\
&P(2n,r/q) = (-1)^{n+1} \frac{(2\pi)^{2n}}{(2n)!} B_{2n} (r/q), \qquad
O(2n-1,r/q) = (-1)^n \frac{(2\pi)^{2n-1}}{(2n-1)!} B_{2n-1}(r/q).
\end{split}
\end{equation*}
\end{proposition}
\begin{proof}
From the functional equation of $P(1-s,a)$ in Lemma \ref{lem:fe1}, we have
$$
P(1-2n,a) = \frac{2(2n-1)!}{(2\pi)^{2n}} \cos (\pi n) Z(2n,a) = (-1)^{n} \frac{2(2n-1)!}{(2\pi)^{2n}} Z(2n,a) .
$$
Thus we obtain the first formula of this proposition from
$$
P(1-2n,r/q) = -\frac{2 q^{2n-1}}{2n} \sum_{m=1}^q \cos (2\pi rm/q) B_{2n} (m/q)
$$
which is proved by Proposition \ref{pro:raZYPOn}. Similarly, one has
$$
O (2-2n,a) = \frac{2(2n-2)!}{(2\pi)^{n-1}} \sin \Bigl( \frac{2n-1}{2} \pi \Bigr) Y(2n-1, a)
$$
by the functional equation of $O(1-s,a)$ in Lemma \ref{lem:fe1}. Hence, we have the second equation of Proposition \ref{pro:raZYPOp} and
$$
O (2-2n,a) = -\frac{2 q^{2n-2}}{2n-1} \sum_{m=1}^q \sin (2\pi rm/q) B_{2n-1} (m/q)
$$
derived from Proposition \ref{pro:raZYPOn}. 

By the definition of $P(s,a)$, it holds that
\begin{equation}\label{eq:PB1}
P(2n, a) = {\rm{Li}}_{2n} (e^{2\pi ia}) + {\rm{Li}}_{2n} (e^{2\pi i(1-a)}) = 2 \sum_{m=1}^\infty \frac{\cos 2\pi ma}{m^{2n}} .
\end{equation}
In addition, we have
\begin{equation}\label{eq:OB1}
O(2n-1, a) = \frac{1}{i} \Bigl( {\rm{Li}}_{2n-1} (e^{2\pi ia}) - {\rm{Li}}_{2n-1} (e^{2\pi i(1-a)}) \Bigr) 
= 2 \sum_{m=1}^\infty \frac{\sin 2\pi ma}{m^{2n-1}} 
\end{equation}
from the definition of $O(s,a)$. Hence, the third and fourth equations in this proposition are prove by (\ref{eq:PB1}), (\ref{eq:OB1}) and Lemma \ref{lemB1}.
\end{proof}

We can immediately show the following by the propositions above and the definitions of $Q(s,a)$ and $X(s,a)$. 
\begin{corollary}\label{cor:QXn}
Let $r,q \in {\mathbb{N}}$ be coprime. Then, for $n \in {\mathbb{N}}$, one has
\begin{equation*}
\begin{split}
2Q(-n,r/q) &= \frac{(-1)^n-1}{n+1} B_{n+1}(r/q) - \frac{2 q^n}{n+1} \sum_{m=1}^q \cos (2\pi rm/q) B_{n+1} (m/q), \\ 
2X(-n,r/q) &= \frac{(-1)^{n+1}-1}{n+1} B_{n+1}(r/q) - \frac{2 q^n}{n+1} \sum_{m=1}^q \sin (2\pi rm/q) B_{n+1} (m/q).
\end{split}
\end{equation*}
\end{corollary}

\begin{corollary}\label{cor:QXp}
Let $r,q \in {\mathbb{N}}$ be coprime. Then, for $n \in {\mathbb{N}}$, one has
\begin{equation*}
\begin{split}
2Q(2n,r/q) &= (-1)^{n+1} \frac{(2\pi)^{2n}}{(2n)!} \biggl( B_{2n} (r/q) + q^{2n-1} \sum_{m=1}^q \cos (2\pi rm/q) B_{2n} (m/q) \biggr),  \\ 
2X(2n-1,r/q) &= (-1)^n \frac{(2\pi)^{2n-1}}{(2n-1)!} \biggl( B_{2n-1}(r/q) + q^{2n-2} \sum_{m=1}^q \sin (2\pi rm/q) B_{2n-1} (m/q) \biggr) .
\end{split}
\end{equation*}
\end{corollary}

\begin{proof}[Proof of Theorem \ref{th:Ber}]
We can prove Theorem \ref{th:Ber} from Propositions \ref{pro:raZYPOn} and \ref{pro:raZYPOp}, Corollary \ref{cor:QXn} and \ref{cor:QXp} and de Moivre's identity
$$
\cos n\theta + i \sin n \theta = (\cos \theta + i \sin \theta)^n, \qquad n \in {\mathbb{N}}, \quad \theta \in {\mathbb{R}},
$$
and fact that Bernoulli polynomials are polynomials with rational coefficients. 
\end{proof}

\section{Proof of Theorem \ref{th:poly}}

\subsection{Generalized Euler polynomials}

For $0 < a < 1$, we define the generalized Euler polynomial $E_{c,n} (t)$ by
$$
\frac{(1+c) e^{tz}}{e^z+c} = \sum_{n=0}^\infty E_{c,n}(t) \frac{z^n}{n!}, \qquad c := -\exp (2\pi ia).
$$
The polynomial $E_{c,n} (t)$ above is introduced in \cite[Section 4.1]{Shi}. Note that similar polynomials are defined by Apostol \cite{ApoL} and Frobenius \cite{Fro}. For simplicity, we put $b := - (1+c)^{-1}$. Then we have (see \cite[Section 4.1]{Shi})
$$
E_{c,n} (t) = t^n + b \sum_{k=0}^{n-1} \binom{n}{k} E_{c,n} (t), \qquad
\frac{d}{dt} E_{c,n} (t) = n E_{c,n-1} (t), \qquad n>0, 
$$
$$
E_{c,n} (t+1) + cE_{c,n} (t) = (1+c)t^n, \qquad E_{1,2n+1}(1/2) =0, 
$$
$$
E_{c,n} (1-t) = (-1)^n E_{c^{-1}\!, \!\; n} (t), \qquad E_{c^{-1}\!, \!\; n} (0) = (-1)^{n+1} c E_{c,n} (0).
$$
For instance, one has
\begin{equation*}\begin{split}
&\quad E_{c,0} (t) = 1, \qquad E_{c,1} (t) = t+b, \qquad E_{c,2} = t^2+2bt+2b^2+b, \\
&E_{c,3} (t) = t^3 + 3bt^2 + (6b^2+3b)t + 6b^3 + 6b^2 +b, \qquad b := - (1+c)^{-1}.
\end{split}\end{equation*}
When $n \in {\mathbb{N}}$ and $0<a<1$, we define $F_n (a)$ by
$$
F_n (a) := \sum_{l \in {\mathbb{Z}}} \frac{1}{(l+a)^{n+1}} = 
\sum_{l = 0}^\infty \frac{1}{(l+a)^{n+1}} + (-1)^{n+1} \sum_{l = 0}^\infty \frac{1}{(l+1-a)^{n+1}}.
$$
We have the following by $E_{c,n} (0) = (1+c^{-1}) n! (2\pi i)^{-n-1} F_n (a)$ proved in \cite[Theorem 4.2]{Shi}. 
\begin{lemma}\label{lem:P1}
For $n \in {\mathbb{N}}$, it holds that
\begin{equation}\label{eq:SKKY1}
F_n (a) = \frac{(2\pi i)^{n+1} E_{c,n} (0)}{n! (1+c^{-1})}, \qquad c := -\exp (2\pi ia).
\end{equation}
\end{lemma}

By using Yamamoto's formula (see \cite[Proposition 3.2]{Yam} or \cite[p.~17]{KKY}) and the functional equation of ${\rm{Li}}_{1-s} (e^{2\pi ia})$ (see Lemma \ref{lem:fe1}), we have the following.
\begin{lemma}\label{lem:P2}
For every integer $n \ge0$, it holds that
\begin{equation}\label{eq:SKKY2}
{\rm{Li}}_{-n} (e^{2\pi ia}) =  
\sum_{r=0}^n \frac{r! (-c)^{r} S(n,r)}{(1+c)^{r+1}} = \frac{n! F_n (1-a)}{(2\pi i)^{n+1}} = \frac{E_{c^{-1}\!, \!\; n} (0)}{1+c} ,
\end{equation}
where $c := -\exp (2\pi ia)$ and $S(n,r)$ is the Stirling numbers of the second kind which is defined as $r! S(n,r) := \sum_{m=1}^r (-1)^{r-m} \binom{r}{m} m^n$.
\end{lemma}
\begin{proof}
From the Yamamoto formula proved in \cite[Proposition 3.2]{Yam}, we have
$$
{\rm{Li}}_{-n} (e^{2\pi ia}) = \sum_{r=0}^n  \frac{r! S(n,r) ( e^{2\pi i a})^r}{(1-e^{2\pi i a})^{r+1}} =  
\sum_{r=0}^n \frac{r! S(n,r) (-c)^{r}}{(1+c)^{r+1}} .
$$
Hence we obtain the first equal sign of (\ref{eq:SKKY2}). By putting $s=n+1$ in the functional equation of ${\rm{Li}}_{1-s} (e^{2\pi ia})$ (see Lemma \ref{lem:fe1}), we have
\begin{equation*}\begin{split}
{\rm{Li}}_{-n} (e^{2\pi ia}) &= \frac{n!}{(2\pi)^{n+1}} \Bigl( i^{n+1} \zeta (n+1,a) + (-i)^{n+1} \zeta (n+1,1-a) \Bigr) \\
&= \frac{(i)^{n+1} n!}{(2\pi)^{n+1}} F_n (a) = \frac{n!}{(2\pi i)^{n+1}} F_n (1-a) 
\end{split}\end{equation*}
which implies the second equal sign of (\ref{eq:SKKY2}). We obtain the third equal sign of (\ref{eq:SKKY2}) from Lemma \ref{lem:P1}.
\end{proof}

\subsection{Proof of Theorem \ref{th:poly}}
In this subsection, we give explicit evaluation formulas for $Z(-n,a)$, $P(2n+2,a)$, $Y(-n,a)$, $O(2n+1,a)$, $Z(2n+2,a)$, $P(-n,a)$, $Y(2n+1,a)$, $O(-n,a)$, $Q(2n+2,a)$, $Q(-n,a)$, $X(2n+1,a)$ and $X(-n,a)$, where $n$ is a non-positive integer, which prove Theorem \ref{th:poly}. 
\begin{proposition}\label{Pro:Z1}
For $n \in {\mathbb{N}}$, we have
$$
Z(2n, a) = \frac{(2\pi i)^{2n} E_{c,2n-1} (0)}{(1+c^{-1}) (2n-1)!}  = 
\frac{c (2\pi i)^{2n}}{(2n-1)!} \sum_{r=0}^{2n-1} \frac{(-1)^r r! S(2n-1,r)}{(1+c)^{r+1}} .
$$
For every integer $n \ge0$, it holds that
$$
Z (-n,a) = - \frac{B_{n+1}(a)+ B_{n+1}(1-a)}{n+1} = \frac{(-1)^n-1}{n+1} B_{n+1} (a).
$$
\end{proposition}
\begin{proof}
Obviously, we have
$$
Z(2n,a) = F_{2n-1} (a). 
$$
Hence, the first formula is proved by (\ref{eq:SKKY1}), (\ref{eq:SKKY2}) and
\begin{equation}\label{eq:ES1}
\frac{E_{c,n} (0)}{1+c^{-1}} = \sum_{r=0}^n \frac{r! (-c^{-1})^r S(n,r)}{(1+c^{-1})^{r+1}} = 
c \sum_{r=0}^n \frac{r! (-1)^r S(n,r)}{(c+1)^{r+1}}. 
\end{equation}
We obtain the second formula from (\ref{eq:ab1-a}) and Lemma \ref{lemB2}. 
\end{proof}

\begin{proposition}\label{Pro:P1}
For $n \in {\mathbb{N}}$, we have
$$
P(2n, a) = (-1)^{n+1} \frac{(2\pi)^{2n}}{(2n)!} B_{2n} (a). 
$$
For every integer $n \ge0$, it holds that
$$
P (-n,a) = \frac{1-(-1)^n}{1+c^{-1}} E_{c,n} (0) = \frac{1-(-1)^n}{c^{-1}+c^{-2}} \sum_{r=0}^n \frac{(-1)^r r! S(n,r)}{(1+c)^{r+1}}.
$$
\end{proposition}
\begin{proof}
We have the first formula from (\ref{eq:PB1}) and Lemma \ref{lemB1}. The second formula is shown by (\ref{eq:ES1}), the definition of $P(s,a)$, Lemma \ref{lem:P2}, the formula
$$
P (-n,a) = \frac{E_{c^{-1}\!, \!\; n} (0)}{1+c} + \frac{E_{c,n} (0)}{1+c^{-1}} 
$$
and the equation $E_{c^{-1},n} (0) = (-1)^{n+1} c E_{c,n} (0)$ (see Section 2.3 or \cite[(4.3g)]{Shi}). 
\end{proof}

\begin{proposition}\label{Pro:Y1}
For $n \in {\mathbb{N}}$, we have
$$
Y(2n-1, a) = \frac{(2\pi i)^{2n-1} E_{c,2n-2} (0)}{(1+c^{-1}) (2n-2)!} = 
\frac{c(2\pi i)^{2n-1}}{(2n-2)!} \sum_{r=0}^{2n-2} \frac{(-1)^r r! S(2n-2,r)}{(1+c)^{r+1}} .
$$
For every integer $n \ge0$, it holds that
$$
Y (-n,a) = - \frac{B_{n+1}(a)- B_{n+1}(1-a)}{n+1} = \frac{(-1)^{n-1}-1}{n+1} B_{n+1} (a).
$$
\end{proposition}
\begin{proof}
The first formula is shown by (\ref{eq:ES1}), Lemma \ref{lem:P1} and
$$
Y(2n-1,a) = F_{2n-2} (a)
$$
if $n \ge 2$. The case $n=1$ is shown by
$$
\lim_{s \to 1} Y(s,a) = \sum_{n=0}^\infty \biggl( \frac{1}{n+a} - \frac{1}{n+1-a} \biggr) =\psi(1-a) - \psi (a) = \pi \cot \pi a,
$$ 
where $\psi (a)$ the digamma function. We have the second formula of this proposition from (\ref{eq:ab1-a}) and Lemma \ref{lemB2} again. 
\end{proof}

\begin{proposition}\label{Pro:O1}
For $n \in {\mathbb{N}}$, we have
$$
O(2n-1, a) = (-1)^n \frac{(2\pi)^{2n-1}}{(2n-1)!} B_{2n-1}(a) . 
$$
For every integer $n \ge0$, it holds that
$$
O (-n,a) = \frac{1+(-1)^n}{i(1+c^{-1})} E_{c,n} (0) = \frac{1+(-1)^n}{i(c^{-1}+c^{-2})} \sum_{r=0}^n \frac{(-1)^r r! S(n,r)}{(1+c)^{r+1}}.
$$
\end{proposition}
\begin{proof}
We have the first formula of this proposition from (\ref{eq:OB1}) and Lemma \ref{lemB1}. We obtain the second formula by Lemma \ref{eq:SKKY2}, the equations 
$$
O (-n,a) = \frac{1}{i} \biggl( \frac{E_{c^{-1}\!, \!\; n} (0)}{1+c} - \frac{E_{c,n} (0)}{1+c^{-1}} \biggr)
$$
and $E_{c^{-1},n} (0) = (-1)^{n+1} c E_{c,n} (0)$ again. 
\end{proof}

By the propositions above and definitions of $Q(s,a)$ and $X(s,a)$, we have the following. 
\begin{corollary}\label{cor:Q1}
For $n \in {\mathbb{N}}$, we have
$$
2Q(2n, a) = \frac{(2\pi i)^{2n} E_{c,2n-1} (0)}{(1+c^{-1}) (2n-1)!} - (-1)^n \frac{(2\pi)^{2n}}{(2n)!} B_{2n} (a).
$$
For every integer $n \ge0$, it holds that
$$
2Q(-n,a) = \frac{1-(-1)^n}{1+c^{-1}} E_{c,n} (0) - \frac{1-(-1)^n}{n+1} B_{n+1} (a).
$$
\end{corollary}
\begin{corollary}\label{cor:X1}
For $n \in {\mathbb{N}}$, we have
$$
2X(2n-1, a) = \frac{(2\pi i)^{2n-1} E_{c,2n-2} (0)}{(1+c^{-1}) (2n-2)!} + (-1)^n \frac{(2\pi)^{2n-1}}{(2n-1)!} B_{2n-1}(a) .
$$
For every integer $n \ge0$, it holds that
$$
2X(-n,a) = \frac{1+(-1)^n}{i(1+c^{-1})} E_{c,n} (0) - \frac{1+(-1)^n}{n+1} B_{n+1} (a) . 
$$
\end{corollary}

\section{Proofs of Theorems \ref{th:d1}{} and \ref{th:d2}}

\subsection{Proof of Theorem \ref{th:d1}}
Recall the Hermite formula
$$
\zeta (s,a) = \frac{a^{-s}}{2} + \frac{a^{1-s}}{s-1} + 2 
\int_0^\infty \frac{\sin (s \arctan (x/a))}{(x^2+a^2)^{s/2} (e^{2\pi x}-1)} dx,
$$
where the integral involved in the formula above converges for all $s \in {\mathbb{C}}$ (see for example \cite[Section 13.2]{WW}). On the other hand, the following equation is well-known:
\begin{equation}\label{eq:1+a}
\zeta (s,a) = a^{-s} + \zeta (s,1+a). 
\end{equation}

\begin{proof}[Proof of Theorem \ref{th:d1} for $Z(s,a)$ and $P(s,a)$]
From (\ref{eq:1+a}), for all $1 \ne s \in {\mathbb{C}}$ with $\sigma >0$, it holds that
$$
|Z(s,a)| \ge a^{-\sigma} - (1-a)^{-\sigma} - |\zeta (s,1+a)| - |\zeta (s,2-a)|.
$$
This inequality and the Hermite formula imply
\begin{equation}\label{asy:Z}
|Z(s,a)| \to \infty, \qquad a \to + 0. 
\end{equation}
Hence for any $1 \ne s \in {\mathbb{C}}$ with $\sigma >0$, there is $0 < a < 1/2$ such that $Z(s,a) \ne 0$. 
Next let $\sigma > 1$. Then we have
\begin{equation*}
\begin{split}
&\int_0^1 Z(1-s,a)^2 da = \biggl( \frac{2\Gamma (s)}{(2\pi )^s} \cos \Bigl( \frac{\pi s}{2} \Bigr) \biggr)^2 \int_0^1 P(s,a)^2 da \\
= & \biggl( \frac{2\Gamma (s)}{(2\pi )^s} \cos \Bigl( \frac{\pi s}{2} \Bigr) \biggr)^2 
\int_0^1 \sum_{m,n=1}^\infty \frac{\cos 2 \pi ma \cos 2 \pi na}{m^s n^s} da = 
 \biggl( \frac{2\Gamma (s)}{(2\pi )^s} \cos \Bigl( \frac{\pi s}{2} \Bigr) \biggr)^2 \frac{\zeta (2s)}{2}
\end{split}
\end{equation*}
from the functional equation of $Z(1-s,a)$ in Lemma \ref{lem:fe1} and the formula 
$$
2 \cos \alpha \cos \beta = \cos (\alpha + \beta) + \cos (\alpha - \beta), \qquad \alpha, \beta \in {\mathbb{R}}. 
$$
According to (\ref{eq:1+a}) and the Hermite formula, the integral $\int_0^1 Z(1-s,a)^2 da$ converges when $\sigma > 1/2$
since one has
$$
\int_0^1 Z(1-s,a)^2 da \ll_s \int_0^1 a^{2s-2} da = O_s (1). 
$$
It is well-known that $\zeta (2s)$ does not vanish when $\sigma > 1/2$ by the Euler product of the Riemann zeta function. Therefore, for any $s \in {\mathbb{C}}$ with $\sigma < 1/2$ and $-s \not \in 2 {\mathbb{N}} \cup \{ 0 \}$, there exists $0 < a < 1/2$ such that $Z(s,a) \ne 0$ from 
\begin{equation*}
\begin{split}
0 \ne & \int_0^1 Z(1-s,a)^2 da = \biggl( \int_0^{1/2} + \int_{1/2}^1 \biggr) Z(1-s,a)^2 da \\ = &
\int_0^{1/2} Z(1-s,a)^2 da + \int_0^{1/2}  Z(1-s,1-a)^2 da = 2 \int_0^{1/2} Z(1-s,a)^2 da
\end{split}
\end{equation*}
which is shown by $Z(s,a) = Z (s,1-a) = \zeta (s,a) + \zeta (s,1-a)$. Thus, we have $Z(s,a) \equiv 0$ for all $0 < a < 1/2$ if and only if $s$ is a non-positive even integer. 

When $\sigma >1$, we have
$$
2 \int_0^{1/2} P(s,a)^2 da = \int_0^1 P(s,a)^2 da = \int_0^1 \sum_{m,n=1}^\infty \frac{\cos 2 \pi ma \cos 2 \pi na}{m^s n^s} da 
= \frac{\zeta (2s)}{2}
$$
from $P(s,a) = P (s,1-a)$. Hence, for any $s \in {\mathbb{C}}$ with $\sigma > 1$, there is $0 < a < 1/2$ such that $P(s,a) \ne 0$. By using the functional equation of $P(1-s,a)$ in Lemma \ref{lem:fe1} and fact proved above that for any $0, 1 \ne s \in {\mathbb{C}}$ with $\sigma > -1$, there is $0< a < 1/2$ such that $Z(s,a)$ does not vanish, we can see that for any $s \in {\mathbb{C}}$ with $\sigma < 2$ and $-s \not \in \{-1,0\} \cup 2{\mathbb{N}}$, there exists $0 < a < 1/2$ such that $P(s,a) \ne 0$. From \cite[(4.12)]{NPCZ}, one has the following equations
$$
P(1,a) = -2 \log (2\sin \pi a ), \qquad P(0,a) =-1. 
$$
Thus, we have $P(s,a) \equiv 0$ for all $0 < a < 1/2$ if and only if $s$ is a negative even integer. 
\end{proof}

\begin{proof}[Proof of Theorem \ref{th:d1} for $Y(s,a)$ and $O(s,a)$]
When $\sigma >0$, we can show $|Y(s,a)| \to \infty$ as $a \to +0$ by modifying the proof of (\ref{asy:Z}). For $\sigma >1$, we have
\begin{equation*}
\begin{split}
&\int_0^1 Y(1-s,a)^2 da = \biggl( \frac{2\Gamma (s)}{(2\pi )^s} \sin \Bigl( \frac{\pi s}{2} \Bigr) \biggr)^2 \int_0^1 O(s,a)^2 da \\
= & \biggl( \frac{2\Gamma (s)}{(2\pi )^s} \sin \Bigl( \frac{\pi s}{2} \Bigr) \biggr)^2 
\int_0^1 \sum_{m,n=1}^\infty \frac{\sin 2 \pi ma \sin 2 \pi na}{m^s n^s} da = 
 \biggl( \frac{2\Gamma (s)}{(2\pi )^s} \sin \Bigl( \frac{\pi s}{2} \Bigr) \biggr)^2 \frac{\zeta (2s)}{2}
\end{split}
\end{equation*}
by the functional equation of $Y(1-s,a)$ in Lemma \ref{lem:fe1} and the equation $2 \sin \alpha \sin \beta = \cos (\alpha - \beta) - \cos (\alpha + \beta)$. It should be mentioned that the integral $\int_0^1 Y(1-s,a)^2 da$ converges when $\sigma > 1/2$ by the Hermite formula and 
$$
\int_0^1 Y(1-s,a)^2 da \ll_s \int_0^1 a^{2s-2} da = O_s (1). 
$$
Furthermore, it holds that
$$
0 \ne \int_0^1 Y(1-s,a)^2 da = 2 \int_0^{1/2} Y(1-s,a)^2 da
$$
by $Y(s,a) = Y (s,1-a)$. Hence we can prove that for any $-s \not \in 2{\mathbb{N}}-1$, there is $0<a<1/2$ such that $Y(s,a) \ne 0$. 

When $\sigma >1$, it holds that
$$
2 \int_0^{1/2} O(s,a)^2 da = \int_0^1 O(s,a)^2 da = \int_0^1 \sum_{m,n=1}^\infty \frac{\sin 2 \pi ma \sin 2 \pi na}{m^s n^s} da 
= \frac{\zeta (2s)}{2}.
$$
Thus, for any $s \in {\mathbb{C}}$ with $\sigma > 1$, there is $0 < a < 1/2$ such that $O(s,a) \ne 0$. From the functional equation of $O(1-s,a)$ in Lemma \ref{lem:fe1} and the fact proved above that for any $s \in {\mathbb{C}}$ with $\sigma > -1$, there is $0< a < 1/2$ such that $Y(s,a)$ does not vanish, we can see that for any $s \in {\mathbb{C}}$ with $\sigma < 2$ and $-s \not \in 2{\mathbb{N}}-1$, there exists $0 < a < 1/2$ such that $O(s,a) \ne 0$. Therefore, we have $O(s,a) \equiv 0$ for all $0 < a < 1/2$ if and only if $s$ is a negative odd integer. 
\end{proof}

\begin{proof}[Proof of Theorem \ref{th:d1} for $Q(s,a)$ and $X(s,a)$]
According to the functional equation of $P(s,a)$ in Lemma \ref{lem:fe1}, one has
\begin{equation}\label{eq:Q0ZP1}
\begin{split}
2Q(s,a) &= Z(s,a) + P(s,a) = Z(s,a) + \frac{(2\pi )^s}{2\Gamma (s) \cos (\pi s/2)} Z(1-s,a) \\
&= Z(s,a) +  2 (2\pi )^{s-1} \Gamma (1-s) \sin \Bigl( \frac{\pi s}{2} \Bigr) Z(1-s,a) .
\end{split}
\end{equation}
From (\ref{eq:1+a}) and the Hermite formula, we have
$$
Z(s,a) = a^{-s} + O_s(1), \qquad Z(1-s,a) = a^{s -1} + O_s(1)
$$
when $a \to +0$, $s \ne 1$ and $\sigma > 1/2$. Hence, for any $1 \ne s \in {\mathbb{C}}$ with $\sigma > 1/2$, there exists $0 < a< 1/2$ such that $Q(s,a) \ne 0$ by (\ref{eq:Q0ZP1}). From (\ref{eq:Q0ZP1}) and the equation $Z(1/2- it,a) = Z(\overline{1/2+ it},a) = \overline{Z(1/2+ it,a)}$, one has
\begin{equation*}
\begin{split}
&\quad \,\, 2Q(1/2+it,a)\\ 
&= Z(1/2+it,a) +  2 (2\pi )^{-1/2-it} \Gamma (1/2-it) \sin \Bigl( \frac{\pi (1+2it)}{4} \Bigr) \overline{Z(1/2+ it,a)} \\
&= a^{-1/2-it} + 2 (2\pi )^{-1/2-it} \Gamma (1/2-it) \sin \Bigl( \frac{\pi (1+2it)}{4} \Bigr) a^{-1/2+it} + O_t(1).
\end{split}
\end{equation*}
Therefore, for any $t \in {\mathbb{R}}$, there exist $0 < a < 1/2$ such that $Q(1/2+it,a) \ne 0$. Hence, for any $1 \ne s \in {\mathbb{C}}$ with $\sigma \ge 1/2$, there exists $0 < a< 1/2$ such that $Q(s,a)$ does not vanish. According to \cite[(2.4)]{NPRCZ}, that we have
$$
Q(0,a) = -1/2 = \zeta (0) \ne 0. 
$$
Thus, by using the functional equation of $Q(s,a)$ in Lemma \ref{lem:fe1}, we have that for any $1 \ne s \in {\mathbb{C}}$ with $-\sigma \not \in 2{\mathbb{N}}$, there exists $0 < a< 1/2$ such that $Q(s,a) \ne 0$. We can similarly show that for any $1 \ne s \in {\mathbb{C}}$ with $-\sigma \not \in 2{\mathbb{N}}-1$, there exists $0 < a< 1/2$ such that $X(s,a)$ does not vanish. 
\end{proof}

\begin{proof}[Proof of Theorem \ref{th:d2}]
Let $1 \ne s \in {\mathbb{C}}$ with $\sigma >0$. Then, there exists $0 < a < 1/2$ such that $\zeta(s,a) \ne 0$ since we have $|\zeta (s,a)| \to \infty$ as $a \to +0$ by modifying the proof of (\ref{asy:Z}). When $\sigma >1$, one has
\begin{equation*}
\begin{split}
\int_0^1 \zeta (1-s,a)^2 da = \frac{2\Gamma (s)^2}{(2\pi)^{2s}} \zeta (2s) 
\end{split}
\end{equation*}
according to the functional equation of $\zeta(1-s,a)$ in Lemma \ref{lem:fe1}. The integral converges absolutely when $\sigma > 1/2$ from (\ref{eq:1+a}) and the Hermite formula. Hence for any $s \in {\mathbb{C}}$ with $\sigma < 1/2$, there is $0 < a < 1$ such that $\zeta(s,a) \ne 0$. In addition, we have
$$
\int_0^{1/2} \zeta (1-s,a) da + \int_{1/2}^1 \zeta (1-s,a) da = \int_0^1 \zeta (1-s,a) da =0 
$$
by the functional equation of $\zeta(1-s,a)$ and $\int_0^1 {\rm{Li}}_s (e^{2\pi ia}) da =0$ for $\sigma >1$. The integral $\int_0^1 \zeta (1-s,a) da$ converges absolutely when $\sigma > 0$ from (\ref{eq:1+a}) and the Hermite formula. Thus, for any $s \in {\mathbb{C}}$ with $\sigma < 1/2$, there is $0 < a < 1/2$ such that $\zeta(s,a)$ does not vanish. 

Suppose $\sigma >1$. Then we have
\begin{equation*}
\begin{split}
&2 \int_0^{1/2} {\rm{Li}}_s (e^{2\pi ia}) {\rm{Li}}_s (e^{2\pi i(1-a)}) da \\
= & \int_0^{1/2} {\rm{Li}}_s (e^{2\pi ia}) {\rm{Li}}_s (e^{2\pi i(1-a)}) da 
+ \int_{1/2}^1 {\rm{Li}}_s (e^{2\pi ia}) {\rm{Li}}_s (e^{2\pi i(1-a)}) da \\
=& \int_0^1 {\rm{Li}}_s (e^{2\pi ia}) {\rm{Li}}_s (e^{2\pi i(1-a)}) da = \zeta (2s).  
\end{split}
\end{equation*}
Hence, for any $s \in {\mathbb{C}}$ with $\sigma > 1$, there is $0 < a < 1/2$ such that ${\rm{Li}}_s (e^{2\pi ia}) \ne 0$. By using (\ref{eq:1+a}), the Hermit formula and functional equation of ${\rm{Li}}_{1-s} (e^{2\pi ia})$ in Lemma \ref{lem:fe1}, we have
$$
{\rm{Li}}_{1-s} (e^{2\pi ia}) = \frac{\Gamma (s)}{(2\pi)^s} \Bigl( e^{\pi i s/2} a^{-s} + O_s (1) \Bigr), \qquad a \to +0
$$
when $\sigma >0$. Hence, for any $s \in {\mathbb{C}}$ with $\sigma < 1$, there is $0 < a < 1/2$ such that ${\rm{Li}}_s (e^{2\pi ia})$ does not vanish. Furthermore, it holds that
$$
\frac{\partial}{\partial a} {\rm{Li}}_s (e^{2\pi ia}) = 2\pi i {\rm{Li}}_{s-1} (e^{2\pi ia}), \qquad 0 < a <1
$$
which implies 
$$
{\rm{Li}}_{1+it} (e^{2\pi ia}) = \frac{1}{2 \pi i} \frac{\partial}{\partial a} {\rm{Li}}_{2+it} (e^{2\pi ia}) , \qquad 0 < a <1.
$$
On the other hand, one has ${\rm{Li}}_{2+it} (1) = \zeta (2+it) \ne 0$ and
$$
{\rm{Li}}_{2+it} (e^{\pi i}) = \bigl( 2^{-1-it} -1 \bigr) \zeta (2+it) \ne {\rm{Li}}_{2+it} (1)
$$
which is proved by
$$
{\rm{Li}}_s (e^{\pi i}) = \sum_{n=1}^\infty \frac{(-1)^n}{n^s} = \frac{-1}{1^s} + \frac{1}{2^s} + \frac{-1}{3^s} + \frac{-1}{4^s} + \cdots
= -\zeta (s) + 2 \cdot 2^{-s} \zeta (s), \qquad \sigma >1 . 
$$
Hence, there is $0 < a < 1/2$ such that $(\partial / \partial){\rm{Li}}_{1+it} (e^{2\pi ia}) \ne 0$ by ${\rm{Li}}_{2+it} (1) \ne {\rm{Li}}_{2+it} (e^{\pi i})$. Therefore, for any  $t \in {\mathbb{R}}$, there exists $0 < a < 1/2$ such that ${\rm{Li}}_{1+it} (e^{2\pi ia}) \ne 0$.
\end{proof}

\begin{remark}
The condition $a \in (0,1/2)$ in Theorems \ref{th:d1} and \ref{th:d2} can be replaced by $a \in I$, where $I \subset (0,1/2)$ is an open interval by the identity theorem and fact that the functions $(s-1)\zeta (s,a)$ and ${\rm{Li}}_s (e^{2\pi ia})$ are real analytic with respect to $a \in (0,1/2)$.
\end{remark}

\section{$Z(s,a)$ and stationary self-similar distribution}
We first define one-dimensional stationary self-similar distributions (see \cite[Section 1]{Sinai}). Let $X$ be the space of realizations of a one-dimensional random field $x:= \{ x_l : l \in {\mathbb{Z}} \}$. Note that each random variable $x$ takes on real values, and the space $X$ is a vector space. There is a group $\{ T_l : l \in {\mathbb{Z}} \}$ of translations acting naturally on the space $X$. The symbols ${\mathfrak{M}}$ and ${\mathfrak{M}}^{\rm{st}}$ denote the space of all probability distributions on $X$ and all stationary distributions on $X$ (namely, distributions invariant with respect to the group $\{ T_l^* : l \in {\mathbb{Z}} \}$ of translations, where $\{ T_l^* : l \in {\mathbb{Z}} \}$ is the group adjoint to $\{ T_l : l \in {\mathbb{Z}} \}$ which acts on ${\mathfrak{M}}$), respectively.

For each $1 < \lambda <2$, we introduce the multiplicative semigroup $A_k (\lambda)= A_k$, where $k \in {\mathbb{N}}$, of linear endmomorphisms of $X$ whose action is given by the formula
$$
\tilde{x_l} = (A_k x)_l := \frac{1}{k^{\lambda/2}} \sum_{lk \le r < (l+1)k} x_r, \qquad l \in {\mathbb{Z}} .
$$
Let $\{ A_k : k \in {\mathbb{N}} \}$ denote the adjoint semigroup acting on the space ${\mathfrak{M}}$, namely, 
$$
(A_k^* P) (C) = P (A_k^{-1} C), \qquad C \subset X, \quad P \in {\mathfrak{M}}. 
$$

\begin{dfn}
A probability distribution $P \in {\mathfrak{M}}$ is called a self-similar distribution (s.d.) if one has
$$
A_k^* P = P \quad \mbox{for all} \quad k \in {\mathbb{N}}.
$$
\end{dfn}
In other words, an s.d.~is a fixed point of the semigroup $\{ A_k^* : k \in {\mathbb{N}} \}$ acting on the space ${\mathfrak{M}}$. On the other hand, It follows from the definition of $A_k$ that $A_k T_{lk} = T_l A_k$. Hence, if  $P \in {\mathfrak{M}}^{\rm{st}}$, then $A_k^* P \in {\mathfrak{M}}^{\rm{st}}$ for any $k \in {\mathbb{N}}$. 
\begin{dfn}
An s.d.~distribution $P \in {\mathfrak{M}}$ is called a stationary self-similar distribution (s.s.d.) if $P \in {\mathfrak{M}}^{\rm{st}}$.
\end{dfn}

Now let $P$ be a one-dimensional stationary Gaussian distribution on $X$ with ${\mathbb{E}} x_l =0$, where ${\mathbb{E}} x_l$ is the expected value of $x_l$. Then we have the following.
\begin{thm}[{\cite[Theorem 2.1]{Sinai}}]
The distribution $P$ is an s.s.d, if and only if its spectral density $\rho_\lambda(\alpha)$ has the form
$$
\rho_\lambda(\alpha):= C\bigl| e^{2\pi i \alpha} -1 \bigr|^2 \sum_{n \in {\mathbb{Z}}} \frac{1}{|n+\alpha|^{\lambda+1}},
\qquad -1/2 \le \alpha \le 1/2, 
$$
where $C>0$ is a constant.
\end{thm}

By the next proposition, we can easily see that the spectral density $\rho_\lambda(\alpha)$ above is written by $Z(\lambda,|\alpha|)$.
\begin{proposition}\label{pro:Si}
When $\sigma >1$ and $\alpha \ne 0$, one has 
$$
\sum_{n \in {\mathbb{Z}}} \frac{1}{|n+\alpha|^s} = Z(s,|\alpha|),
\qquad -1/2 \le \alpha \le 1/2.
$$
\end{proposition}
\begin{proof}
For $0 <  \alpha \le 1/2$, we have
\begin{equation*}
\begin{split}
\sum_{n \in {\mathbb{Z}}} \frac{1}{|n+\alpha|^s} &= 
\sum_{n=0}^\infty \frac{1}{|n+\alpha|^s} + \sum_{n=-1}^{-\infty} \frac{1}{|n+\alpha|^s}
= \sum_{n=0}^\infty \frac{1}{(n+\alpha)^s} + \sum_{n=0}^{\infty} \frac{1}{(n+1-\alpha)^s} \\
&= \zeta (s,\alpha) + \zeta (s,1-\alpha) = Z(s,\alpha) = Z(s,|\alpha|). 
\end{split}
\end{equation*}
When $-1/2 \le  \alpha < 0$, it holds that
\begin{equation*}
\begin{split}
\sum_{n \in {\mathbb{Z}}} \frac{1}{|n+\alpha|^s} &= 
\sum_{n=1}^\infty \frac{1}{|n+\alpha|^s} + \sum_{n=0}^{-\infty} \frac{1}{|n+\alpha|^s}
= \sum_{n=0}^\infty \frac{1}{(n+1-|\alpha|)^s} + \sum_{n=0}^{\infty} \frac{1}{(n+|\alpha|)^s} \\
&= \zeta (s,1-|\alpha|) + \zeta (s,|\alpha|) = Z(s,|\alpha|). 
\end{split}
\end{equation*}
The equations above imply Proposition \ref{pro:Si}.
\end{proof}

\begin{proof}[Proof of Proposition \ref{th:apSi1}]
This is easily proved by Theorem E and Proposition \ref{pro:Si}.
\end{proof}

\begin{remark}
Fukasawa and Takabatake \cite[p.~1877]{FT} considered a sequence of $n$-dimensional centered Gaussian random vectors which covariance functions are characterized by the following spectral density:
$$
\rho^2 \delta_n^{2H} \frac{\Gamma(2H+1) \sin (\pi H)}{(2\pi)^{2+2H+2\psi}} \bigl( 2 - 2 \cos (2\pi \alpha) \bigr)^{\psi +1}
\sum_{n \in {\mathbb{Z}}} \frac{1}{|n+\alpha|^{1+2H+2\psi}},
$$
where $\rho, \psi >0$, $0 < H \le 1$, $n$ is the sample size and $\delta_n$ is the length of sampling intervals. Note that the infinite series above coincides with $Z(1+2H+2\psi,|\alpha|)$ by Proposition \ref{pro:Si}. Moreover, the function $Z(1+2H+2\psi,|\alpha|)$ can be expressed as a rational function with rational coefficients of $\exp (2 \pi i |\alpha|)$ from Proposition \ref{Pro:Z1} (see also Section 3.1).
\end{remark}

\subsection*{Acknowledgments}
The author was partially supported by JSPS grant 16K05077. I would like to express my gratitude to Professor Fukasawa for his comments on Sinai's paper and related topics.\\

 
\end{document}